\documentclass[a4paper,11pt]{amsart}

\usepackage[applemac]{inputenc}
\usepackage{amsmath, amsthm, amssymb}
\usepackage[all]{xy}
\usepackage{marginnote}
\usepackage{color}
\usepackage{mathrsfs}
\usepackage{frcursive}

\theoremstyle{plain}
\newtheorem{theorem}{Theorem}[section]
\newtheorem*{thma}{Theorem A}
\newtheorem*{thmb}{Theorem B}
\newtheorem*{thmc}{Theorem C}

\newtheorem{proposition}[theorem]{Proposition}
\newtheorem{corollary}[theorem]{Corollary}
\newtheorem{lemma}[theorem]{Lemma}
\newtheorem{conjecture}[theorem]{Conjecture}

\theoremstyle{definition}

\theoremstyle{remark}
\newtheorem{remark}[theorem]{Remark}

\DeclareMathOperator{\ord}{ord}
\DeclareMathOperator{\weight}{weight}
\DeclareMathOperator{\card}{card}
\DeclareMathOperator{\End}{End}

\def\R{\mathbb R}
\def\C{\mathbb C}
\def\P{\mathbb P}
\def\O{\mathcal O}
\def\Z{\mathbb Z}
\def\Q{\mathbb Q}
\def\F{\mathcal F}

\def\G{\mathcal G}
\def\H{\mathfrak H}

\begin{document}
\bibliographystyle{alpha}

\title{Curves in Hilbert modular varieties} 

\author{Erwan Rousseau \and Fr\'ed\'eric Touzet }
\address{Erwan Rousseau \\ Institut Universitaire de France \& Aix Marseille Universit\'e, CNRS, Centrale Marseille, I2M, UMR 7373, 13453 Marseille, France.}
\email{erwan.rousseau@univ-amu.fr}
\address{Fr\'ed\'eric Touzet \\ IRMAR, Campus de Beaulieu 35042 Rennes Cedex, France.}
\email{frederic.touzet@univ-rennes1.fr}


\begin{abstract}
We prove a boundedness-theorem for families of abelian varieties with real multiplication.
More generally, we study curves in Hilbert modular varieties from the point of view of the Green-Griffiths-Lang conjecture claiming that entire curves in complex projective varieties of general type should be contained in a proper subvariety. Using holomorphic foliations theory, we establish a Second Main Theorem following Nevanlinna theory. Finally, with a metric approach, we establish the strong Green-Griffiths-Lang conjecture for Hilbert modular varieties up to finitely many possible exceptions.
\end{abstract}

\maketitle

\section{Introduction}
Let us start by recalling the following boundedness-theorem of Faltings \cite{Fa83} on families of abelian varieties. $C$ is a smooth complex algebraic curve, $S \subset C$ a finite set of points and we consider families $p: X \to C \setminus S$ of principally polarized abelian varieties of dimension $g$ with a level $n$ structure ($n\geq 3)$. This gives a map to the corresponding moduli space $\phi: C \to \overline{\mathcal{A}}:=\overline{\H_g / \Gamma_n}$ where $\H_g$ is the Siegel upper half space and $\Gamma_n$ is the congruence subgroup of level $n$ of $Sp(2g, \Z).$ Moreover $\phi^{-1}(D) \subset S$ where $D:= \overline{\mathcal{A}} \setminus \mathcal{A}.$

Let $L:=K_\mathcal{A}+D$ be the logarithmic canonical bundle (which is big and nef). Then it is proved in \cite{Fa83} that there is a constant $c$ such that 
$$\deg \phi^*(L) \leq c$$
for any such $\phi$. In addition, it is shown that one can take $$c=g(g+1)(3g(C)+s+1)$$ where $g(C)$ is the genus of $C$ and $s:= \card S.$ Later Kim \cite{Kim98} has obtained $$c=\frac{g(g+1)}{2}(2g(C)-2+s).$$

This result can be seen as an illustration of general conjectures of Lang and Vojta which should imply the following statement.
\begin{conjecture}[Lang, Vojta]
Let $\overline{X}$ be a complex projective manifold, $D=\overline{X}\setminus X$ a normal crossing divisor, $C$ a smooth projective curve and $S \subset C$ a finite subset. If the pair $(X,D)$ is of log-general type (i.e $K_{\overline{X}}+D$ is big) then there exists a constant $c$ and a proper subvariety $Y \subsetneq \overline{X}$ such that
$$\deg \phi(C) \leq c$$
for all morphisms $\phi: C \to \overline{X}$ such that $\phi^{-1}(D) \subset S$ and $\phi(C) \not \subset Y$.
\end{conjecture}

The first result of this paper provides another illustration of this conjecture. We study the case of families $p: X \to C \setminus S$ of principally polarized abelian varieties equipped with real multiplication i.e. there is an injection $\mathcal{O}_K \hookrightarrow \End A,$ of the ring of integers of a totally real number field of degree $n$. As above, this gives a map $\phi: C \to Y_K$ to the corresponding moduli space.

Let us recall the construction of this moduli space. The Hilbert modular group is defined as $\Gamma_K:=SL_2(\mathcal{O})$ where $\mathcal{O}:=\mathcal{O}_K$ denotes the ring of algebraic integers of $K$. $\Gamma_K$ acts on the product of $n$ upper half-planes $\H^n.$ There is a natural compactification of the quotient $Y_K:=\overline{\H^n/\Gamma_K}$ adding finitely many cusps. There is a projective resolution, called a \emph{Hilbert modular variety}, $\pi: X\rightarrow {\overline{{\H}^n / \Gamma_K}}$ with $E$ the exceptional divisor.

In this setting, we prove the analogue of Faltings' theorem.

\begin{thma}\label{FA}
Let $X$ be a Hilbert modular variety with $E \subset X$ the exceptional divisor, $C$ a smooth projective curve and $S \subset C$ a finite subset. Then
$$\deg_C(f^*(K_X+E)) \leq n(2g(C)-2+s),$$
for all morphisms $f: C \to X$ such that $f^{-1}(E) \subset S$.
\end{thma}

In the sequel, we turn to the analytic side of these questions.
The geometry of Hilbert modular varieties has attracted a lot of attention, especially in dimension $n=2$ \cite{Hirz}. In higher dimension, the Kodaira dimension of Hilbert modular varieties was studied in \cite{Tsu85} showing that for almost all $K$, $X_K$ is of general type. This property suggests that Hilbert modular varieties should in general satisfy the Green-Griffiths-Lang principle.

\begin{conjecture}(Strong Green-Griffiths-Lang conjecture)
Let $X$ be a complex projective variety of general type. Then there exists a proper algebraic subvariety $Z\subsetneq X$ such that every (non-constant) entire curve $f:\C \to X$ satisfies $f(\C) \subset Z$.
\end{conjecture}

The geometry of algebraic curves from this point of view started in \cite{Frei80} where it is proved that, for principal congruence subgroups of sufficiently high level, rational and elliptic curves all lie in the boundary divisor of the compactification.

This statement was later generalized by \cite{Na} who deals with transcendental curves and quotients of bounded symmetric domains by arithmetic lattices.

Hilbert modular varieties have also been studied by people interested in holomorphic foliations theory. Indeed, they carry canonical holomorphic foliations (induced by the natural foliations living on the product of upper half-planes)  with interesting properties such as negative Kodaira dimension. These properties have been recently used in \cite{DiRou} to prove that Hilbert modular varieties provide counter-examples to the so-called ``jet differentials'' strategy developed (by Bloch, Green-Griffiths, Demailly, Siu...) to attack the Green-Griffiths-Lang conjecture.

One of the goal of this paper is to use alternative methods which do not need to pass to sufficiently high level and in particular, get rid of the torsion. Indeed, such phenomenon has important consequences: considering the diagonal embedding $\H \to \H^n$ and taking the quotient $\overline{\H/SL_2(\Z)} \to \overline{{\H}^n / \Gamma_K}$ show that there are always rational curves not contained in the boundary divisor.

Here, we use foliation theory to obtain a second main theorem (in the language of Nevanlinna theory) for entire curves into Hilbert modular varieties, generalizing previous works \cite{Tib} which dealt with the surface case.

\begin{thmb}
There is a  projective resolution $\pi: X\rightarrow {\overline{{\H}^n / \Gamma_K}}$ with $E$ the exceptional divisor, $K_X$ the canonical line bundle of $X$ such that if $f: \C \to X$ is a non-constant entire curve not contained in $E$ then
$$T_f(r, K_X)+T_f(r, E) \leq nN_1(r, f^*E)+S_f(r) \|,$$
where $S_f(r)=O(\log^+ T_f(r))+o(\log r),$ and $\|$ means that the estimate holds outside some exceptional set of finite measure.
\end{thmb}

This has some consequences on the geometry of curves ramifying on the boundary divisor.

\begin{corollary}
Let $X$ be a Hilbert modular variety of general type as above. Let $f: \C \to X$ be a non-constant entire curve which ramifies  over $E$ with order at least $n$, i.e. $f^*E\geq n \operatorname{supp} f^*E$. Then $f(\C)$ is contained in $E$.
\end{corollary}

In the third part, we develop a metric approach initiated in \cite{Rou13}. Using Hilbert modular forms and the Begmann metric we are able to construct pseudo-metrics which give control on the Kobayashi pseudo-distance. In particular, we prove 

\begin{thmc}
Let $n\geq 2$. Then, except finitely many possible exceptions, Hilbert modular varieties of dimension $n$ satisfy the strong Green-Griffiths-Lang conjecture.
\end{thmc}

\subsubsection*{Acknowledgments}

The first--named author would like to warmly thank Amir D\v{z}ambi\'c and Carlo Gasbarri for interesting discussions about the arithmetic aspects of this paper.

\section{A tangency formula}

\subsection{Tangential locus of foliations}

Let $X$ be a $m$-dimensional complex manifold equipped with $m$ codimension $1$ holomorphic (maybe singular) foliations ${\F}_1,...,{\F}_m$ in general position, which means that, given $m$ locally defining $1$-forms $\omega_1,..., \omega_m$ at $p\in X$ of ${\F}_1,..{\F}_m$, one has 

    $$\omega_1\wedge\omega_2\wedge.....\wedge\omega_m\not\equiv 0$$
with $\mbox{codim}\ (\mbox{Sing}\ \omega_i)\geq 2$.

Let $\eta$ be a local non-degenerate holomorphic volume form and $f$ a holomorphic function such that  $\omega_1\wedge\omega_2\wedge.....\wedge\omega_m=f\eta,$ and $H_p$ the (local) divisor defined by the (possibly non reduced) equation $f=0$.

Note that $H_p$ (as a germ) does not depend on the particular choice of the $\omega_i$'s and $\eta$ and this allows us to define the tangential divisor $H=\mbox{tan}({\F}_1,..., {\F}_m)$: this is nothing but the integral effective divisor on $X$ whose germ at $p\in X$ is $H_p$.

Here is a list of basic properties satisfied by this divisor:

\begin{enumerate}

 \item $N_{{\F}_1}^*\otimes...\otimes N_{{\F}_m}^*=K_X\otimes{\mathcal O}(-H)$.

\item $H$ contains the union of singular loci $\bigcup_i\mbox{Sing}\ {\F}_i$ in its support.

\item If $K$ is an invariant hypersurface for each foliation ${\F}_i$ , then $(m-1)K\leq H$.
\end{enumerate}

\subsection{Some examples.}

\subsubsection{Resolution of cyclic singularities}

This problem has been adressed by Fujiki (\cite{fu}). We will consider a particular case, namely the type of cyclic singularity occuring on Hilbert modular varieties.

Let $G$ a finite cyclic group with a fixed generator acting on a complex affine $m$-space ${\C}^ m={\C}^m(z_1,...,z_m)$ by the formula:

$$g(z_1,..,z_m)=(e_1 z_1,...,e_m z_m)$$

where, for each $i$, $e_i$ is a $n^{th}$ root of the unity, $n>1$. Then the quotient $X={\C}^m/G$ has the structure of a normal affine algebraic variety whose singular locus is reduced to the single point $p=\pi (0)$ where $\pi:{\C}^m\rightarrow X$ denotes the canonical projection.

Following \cite{fu} (theorem 1 p.303) there exists a resolution of the singularity of $X$, that is a pair $(\tilde X, f)$ consisting of a smooth variety $\tilde X$, a proper birational morphism $f:\tilde X \rightarrow X$ isomorphic outside $f^{-1} (p)$ with the additional following properties 

\begin{enumerate}
 \item $E=f^{-1} (p)$ is a simple normal crossing hypersurface, each component of which being rational.

\item There exists an affine open covering $\mathcal U=(U_1,...,U_l)$ of $\tilde X$ with $U_i\simeq{\C}^m$.

\item $f^{-1}(p)\cap U_i$ is defined by $u_{k_1}^i...u_{k_t}^i=0$ in $U_i\simeq {\C}^m(u^i)$ for some $k_1,..,k_t$ (depending on $i$) if it is non empty.

\item For each $i$ the multivaluate map $f\circ \pi:{\C}^m (u^i)\rightarrow {\C}^m(z)$ takes the following form:
 \begin{equation}\label{transform1}
z_k=\prod_{l=1}^m {u_l^i}^{b_{lk}}
\end{equation}

where $b_{lk}$ are nonnegative rational numbers (depending on $i$).
\end{enumerate}

${\C}^m={\C}^m (z)$ comes equipped with $m$ tranverse codimension $1$ foliations ${\mathcal G}_i$ defined respectively by $dz_i=0, i=1,...,m$. Each of them  gives  rise on the quotient $X$ to a foliation ${\mathcal H}_i$. 

Let us set ${\F}_i = f^*{\mathcal H}_i$.
 Remark that ${z_i}^n$ descends to $X$ as an holomorphic first integral $F_i$ of ${\mathcal H}_i$. Consequently $G_i=F_i\circ f$ is a holomorphic first integral for ${\F}_i$ vanishing on $E$, whence the following lemma.
\begin{lemma}
 $E$ is invariant by each ${\F}_i$.
\end{lemma}

\begin{lemma}\label {tangencyorder}
 The tangential divisor $$H=\mbox{tan}({\F}_1,..., {\F}_m)$$ is equal to $(m-1)E$.
\end{lemma}
\begin{proof}
 As $E$ is ${\F}_i$ invariant, one can assert that $H\geq (m-1)E$ and that equality holds if and only $E$ is contained in the polar locus of $$\Omega=\frac{dG_1}{G_1}\wedge....\wedge\frac{dG_m}{G_m}.$$ On the other hand, using the expression (\ref{transform1}), one can deduce that $$\Omega_{|U_i}=\lambda_i \frac{d u_1^i}{u_1^i}\wedge....\wedge \frac{d u_m^i}{u_m^i}$$ where $\lambda_i$ is a positive rational number, thus providing the desired equality. 

\end{proof}

\begin{remark}
One can recover  lemma \ref{tangencyorder}  using  the language of equivariant toric desingularization see \cite{mum} (and even more generally using Hironaka's equivariant resolution of singularities). It 's likely that this is more or less equivalent to Fujiki's construction while having the advantage of being explicitly {\it projective} (this is not totally clear concerning Fujiki's paper).

Indeed, denote also by $f:\tilde X\rightarrow X$ an equivariant toric resolution of $X={\C}^m/G$  (whose underlying torus is ${({\C}^*)} ^m/G$) obtained by blowing-up at $p$ and let $E$ be the exceptional divisor. 

As previously, we set ${\F}_i = f^*{\mathcal H}_i$, $H= \mbox{tan}({\F}_1,..., {\F}_m) $,  $G_i=F_i\circ f$. One more time, it is sufficient to prove that the polar locus of $\Omega=\frac{dG_1}{G_1}\wedge....\wedge\frac{dG_m}{G_m}$ contains each branch of $E$. This can be seen as follows: consider the standard torus $T=
{(\C^*)}^m$ whose action on itself is infinitesimally generated by the $m$ vector fields $Z_1=z_1\frac{\partial} {\partial z_1},...,Z_m= z_m\frac{\partial} {\partial z_m}$. Remark that the action of the torus ${({\C}^*)} ^m/G$ on $X$ is generated by $\pi_* Z_i$ $i=1,...,m$. Set $\xi= \pi_* (\frac{d z_1}{z_1}\wedge....\wedge \frac{d z_m}{z_m})$ and note that the logarithmic form $\frac{dF_i}{F_i}$ is obtained from $\xi$ by contraction with the (holomorphic) $m-1$ vector field $Z_2\wedge...\wedge Z_m$. By equivariance and obvious functoriality properties, the logarithmic form $\frac{dG_1}{G_1}$, whose polar locus contains $E$, can be constructed contracting $\Omega$ by an {\it holomorphic} $m-1$ vector field. This implies that the polar locus of $\Omega$ contains $E$.

\end{remark}

\subsubsection{Resolution of  cusps of ${\H}^m / \Gamma$.}\label{resolutionofcusps}

Let ${\H}^m$ be the product of $m$ copies of the Poincar\'e's upper half plane $\H$. Let $\Gamma$ be an irreducible lattice of the connected component of $G=\mbox{Aut}({\H^m})$ ($\simeq{ \mbox{Aut}(\H)}^m$). Following Shimizu \cite{shi}, there are finitely many  cusps  (wich give rise to additional points in the Baily-Borel compactification   ${\overline{{\H}^m / \Gamma}}^{BB}$ ) and the stabilizer to each of them can be represented (up to congugation in $G$) by 

$$\Gamma (N,V)=N\rtimes V=\Big\{\left(\begin{array}{cc}v & n \\0 & 1\end{array}\right)\in \mbox{GL}_2 (K)\arrowvert v\in V, n\in N \Big\}.$$

where

\begin{itemize}

\item $K$ is a totally real extension of degree $m$ over $\mathbb Q$.

\item $N\subset K$ is a free $\mathbb Z$ module of rank $m$.

\item $V$ is a finite index subgroup of $U_N^+$, the group of totally positive units that satisfy $uN=N$.
\end{itemize}

The group $\Gamma (N,V)$ acts on ${\H}^n$ by affine transformations

$$(z_j)\rightarrow v^{(j)}z_j + n^{(j)}$$

where $k\rightarrow k^{(j)}$ denotes the $j^{th}$ embedding of $K$ in $\mathbb R$.

We can find a neighborhood of the distinguished cusp $\{\infty\}$ such that the Baily-Borel compactification looks like ${\H}^n/\Gamma (N,V)\cup \{\infty\}$ (with a natural structure of normal analytic space singular at $\{\infty\}$ for a suitable pair $(N,V)$. We know focus on Ehler's procedure (without going into details) to resolve the cusp singularity by means of toroidal compactification (a particular case of the general machinery developped in \cite{mum2}). 

 Consider first the natural inclusion ${{\H}^m}/N\hookrightarrow T:={\mathbb C}^m/N$, the action of $N$ is defined as above. Using partial polyhedral cone decomposition (which depends on combinatorial data), one can then construct an equivariant embedding of the torus $T$ into a smooth analytic  manifold $X=\bigcup {(\mathbb C)}_\sigma^m/\sim$ obtained by glueing copies of ${\mathbb C}^m$ (as there are infinitely many of them, $X$ is not properly speaking a toric variety). Moreover $X$ contains a normal crossing hypersurface $F$ whose trace on each copy ${(\mathbb C)}_\sigma^m$ is the union of the coordinates hyperplanes and such that in the coordinates $u$ corresponding to some ${(\mathbb C)}_\sigma^m$, the map 
 
 $$\pi:X- F\rightarrow T$$
 
 that takes 
 
 $$\pi (u)= \left(\begin{array}{c}\mu_1^{(1)}\log{u_1}+...\mu_m^{(1)}\log {u_m} \\\vdots \\\mu_1^{(m)}\log{u_1}+...\mu_m^{(m)}\log {u_m}\end{array}\right)$$
 
 yields an isomorphism 
 
 $$X-F\simeq T$$
 
 where $\mu_1,...,\mu_m$ are totally positive elements of $N$.
 
 By lemmata 2 in \cite{eh}, the group $V$ acts freely and discontinuously on the open neighborhood of $F$, $\tilde X= \pi^{-1} ({{\H}^m}/N)\cup F $,  fixes globally $F$ and is cocompact with respect to $F$. It follows that $Y=\tilde X/V$ is a complex manifold and moreover one can prove that the isomorphism $\pi$ extends on $Y$ as an holomorphic map
 
      $$\hat{\pi}:Y\rightarrow  {\H}^m/\Gamma (N,V) \cup \{\infty\}$$
      
      that takes $F/V$ to $\{\infty\}$.
      
      Let ${\G}_1,..., {\G}_m$ the $m$ tautological codimension one foliation on  ${\H}^m/\Gamma (N,V) \cup \{\infty\}$      
  induced by $dz_1,...., dz_m$ on  ${\H}^m$. For $i=1,...,m$ denote by ${\mathcal F}_i$  the pull-back foliation ${\hat{\pi}}^*{\G}_i$.
  
  \begin{lemma}
  
  Let $E$ be the normal crossing compact hypersurface $F/V$.
    Then the tangential divisor
$$H=\mbox{tan}({\F}_1,..., {\F}_m)$$ is equal to $(m-1)E$.    
  \end{lemma}
  
  \begin{proof}
  Each foliation ${\F}_i$ lifts to a foliation ${\mathcal H}_i$ on $\tilde X$ via the covering map ${\tilde X}\rightarrow Y$.
  It is clearly sufficient to prove that  $\mbox{tan}({\mathcal H}_1,..., {\mathcal H}_m)$ is equal to $(m-1)F$. This is straighforward once we have observed that ${\mathcal H}_i$ is defined by the logarithmic form    
  
  $$\mu_1^{(i)}\frac{du_1}{u_1}+...\mu_m^{(i)}\frac{d u_n }{u_n}$$     
  
  and that the matrix $(\mu_i^{(j)})$ is invertible.
  
  \end{proof}
  
  According to \cite{mum2} IV.2, this toroidal compactification can be realized by some sequence of blowing-up for suitable combinatorial data. 
  
  \subsection {The tangency formula}. Putting together the different informations collected above, one can state the 
  
  \begin{theorem}\label{TAN}
   Let $\Gamma$ be an irreducible lattice of the connected component of $\mbox{Aut}({\H^m})$. Let $p_1,...,p_l$ the singular points arising from the Baily-Borel compactification  ${\overline{{\H}^m / \Gamma}}^{BB}$ (one also includes orbifold singularities). For $i=1,...,m$, let ${\mathcal G}_i$ be the codimension one foliation on ${\overline{{\H}^m / \Gamma}}^{BB}$ obtained by projecting the $i^{th}$ tautological foliation on ${\H}^m$ given by $dz_i=0$. Then there exists a  projective resolution $\pi: X\rightarrow {\overline{{\H}^m / \Gamma}}^{BB}$  of these singularities such that 
   
    $$N_{{\F}_1}^*\otimes...\otimes N_{{\F}_m}^*=K_X\otimes{\mathcal O}\big(-(m-1)E\big)$$  
    
where $E=\pi^{-1} (\{p_1,...,p_l\})$ is the exceptional divisor (a normal crossing hypersurface) and ${\F}_i= {{\pi}}^* {\mathcal G}_i$.

   \end{theorem}

\section{Leafwise hyperbolicity of Hilbert modular foliations}

We take the same notations as above. Let ${\mathcal G}_i$ be the codimension one foliation on ${\overline{{\H}^m / \Gamma}}$ obtained by projecting the $i^{th}$ tautological foliation on ${\H}^m$ given by $dz_i=0$, a  projective resolution $\pi: X\rightarrow {\overline{{\H}^m / \Gamma}}$, $E=\pi^{-1} (\{p_1,...,p_l\})$ is the exceptional divisor and ${\F}_i= {{\pi}}^* {\mathcal G}_i$.

We shall investigate the (Brody) hyperbolicity of the leaves of ${\mathcal G}_i$. In the case of surfaces, the hyperbolicity of the leaves of these foliations is already stated in \cite{MP} and \cite{Tib} although both proofs are not complete. We generalize these statements and prove indeed that

\begin{proposition}\label{LWHYP}
The leaves of the Hilbert modular foliations ${\mathcal G}_i$ are hyperbolic.
\end{proposition}

To prove this, we need a lemma.

\begin{lemma}\label{hypdisk}
Let $\Gamma \subset \mbox{Aut}({\H^m})$ be a finite group. Then ${\H}^m / \Gamma$ is hyperbolic.
\end{lemma}

\begin{proof}
Consider the polydisc $\Delta^m \simeq {\H}^m$.
Suppose there is a non-constant holomorphic map $f: \C \to {\Delta}^m / \Gamma$ and take $x \neq y$ two different points on $f(\C) \subset {\Delta}^m / \Gamma$. Now, we construct a bounded holomorphic function on ${\Delta}^m / \Gamma$ which is $0$ on $x$ and not $0$ on $y$. Let $x_i \in {\Delta}^m$ and $y_k \in {\Delta}^m$ be the preimages of $x$ and $y$ under the projection ${\Delta}^m \to {\Delta}^m / \Gamma$. Take a linear function $h(z)$ on ${\Delta}^m$  such that the set $\{h(z)=0\}$ passes through $x_1$ but avoids the points $y_k$. Then $g(z):= \prod_{\alpha \in \Gamma} h(\alpha.z)$ is invariant under $\Gamma$. It gives a holomorphic function $G$ on ${\Delta}^m / \Gamma$ such that $G(x)=0$ but $G(y) \neq 0$. Then $G\circ f$ would give a non-constant bounded holomorphic function on $\C$. This is a contradiction.
\end{proof}

Now, we prove the proposition

\begin{proof}
We have the following dichotomy: either the leaf passes through an orbifold point or not. In the last case, the stabilizer of the leaf in $\H^m$ has no fixed point and hyperbolicity is immediate. In the first case, by the preceding lemma, it is sufficient to prove that the stabilizer is finite. One observes that the action of the Hilbert modular group on $\H^m$ implies that this stabilizer is a commutative group. Indeed, the Hilbert modular group of the totally real number field $K$, $\Gamma_K=SL(2, \mathcal O)$ acts on $\H^m$ via the embedding of groups $SL(2,K) \hookrightarrow SL(2, \R)^m.$ We consider now the projections $p_i: SL(2, \R)^m \to SL(2, \R)$ whose restrictions to $\Gamma_K$ are injective. If two elements $g$ and $h$ of $\Gamma_K$ are in the stabilizer of a leaf, it means that the projections $g_i:=p_i(g)$ and $h_i:=p_i(h)$ of $g$ and $h$ on the corresponding factor of $SL(2, \R)^m$ have the same fixed point and so commute. This implies that $g$ and $h$ must commute. Therefore all projections $g_k$ and $h_k$ commute. Now suppose that $g$ is in the stabilizer of the orbifold point, which is assumed to be non trivial. Then any other element $h$ of the stabilizer of the leaf commutes with $g$ which means that $h$ is in fact in the stabilizer of the orbifold point. We have proved that the stabilizer of the leaf and of the orbifold point coincide, which concludes the proof.
\end{proof}

\begin{remark}\label{Stein}
As we have just seen, leaves passing through an elliptic fixed point are Stein and hyperbolic. In particular, they do not contain algebraic curves. This is also true for other leaves.

\end{remark}

 \begin{proposition}\label{NC}
The leaves of the Hilbert modular foliations ${\mathcal G}_i$ do not contain algebraic curves.
\end{proposition}

\begin{proof}
We just have to consider the case of leaves which do not pass through an elliptic fixed point. Consider a leaf of the foliation given by $dz_1=0$ in $\H^m$. As explained above, the stabilizer $\Gamma_1$ of this leaf is a commutative group. Therefore the projections $p_i(\Gamma_1)$ consist of automorphisms of $\H$ of the same type (elliptic, hyperbolic or parabolic) and with the same fixed points in $\overline{\H}$. Observe that if $\Gamma_1$ is not trivial there is no parabolic projection $p_i(\Gamma_1)$ because conjugates of parabolic elements are still parabolic. As the leaf does not pass through an elliptic fixed point, there is a projection which is not elliptic. Suppose that the second projection $p_2(\Gamma_1)$ is hyperbolic. Therefore, up to conjugation, $p_2(\Gamma_1)$ consist of hyperbolic elements $\gamma_2: z_2 \to \lambda z_2$ where $\lambda \in \R^+.$ So, the form $\omega=\frac{dz_2}{z_2}$ is invariant under $\Gamma_1$ and gives a holomorphic form on the corresponding leaf of ${\mathcal G}_1$. Now, observe that the periods of $\omega$ are of the form $\log \lambda$, in particular real. Therefore, given an algebraic curve $C$ in the leaf, $\omega$ has to vanish when restricted to the curve. In particular, such a curve has to be also contained in a leaf of the foliation given by $dz_2=0$.

Now suppose that the second projection $p_2(\Gamma_1)$ is elliptic. Taking the model of the disc, $p_2(\Gamma_1)$ consist of rotations $\gamma_2: z_2 \to \zeta z_2$ where $|\zeta|=1.$ Therefore $|z_2|$ is invariant under $\Gamma_1$ and gives a plurisubharmonic function on the leaf. It has to be constant on any algebraic curve. Again, such a curve has to be contained in a leaf of the foliation given by $dz_2=0$.

This true for all $i$, so there cannot be any algebraic curve in a leaf.
\end{proof}

Now, let us a see a consequence of proposition \ref{LWHYP}.

\begin{corollary}\label{entirecurvetangent}
  Let $f:\mathbb C\rightarrow X$ an entire curve tangent to one of the $m$ tautological foliations ${\mathcal F}_i$ on $X$. Let $E$ the exceptional divisor arising from the desingularization of the set of singular points $\{p_1,...,p_l\}$. Assume that $f(\mathbb C)\not\subset E$. Then, $f$ is constant.
\end{corollary}

\begin{proof}
From the proposition \ref{LWHYP}, it is sufficient to prove that such curves $f$ avoid the hypersurface $E_c\subset E$ corresponding to the cusps resolution.
As  described in section \ref{resolutionofcusps}, the change of coordinates in the resolution of cusps is (locally) given by
$$z_i=\sum_j a_{ji} \log u_j.$$ where the coefficients $a_{ij}$ are {\it positive } real numbers. In particular the (locally defined)  real valued function 

$$g=\prod_j{|u_j|}^{a_{ij}}$$ is continuous, constant along the leaves of ${\mathcal F}_i$ and vanishes precisely on $F$. Let $\mathscr C$ be a germ of analytic curve tangent to ${\mathcal F}_i$ passing through $p\in E_c$. For obvious continuity reasons, $g$ necessarily vanishes along $\mathscr C$, hence $\mathscr C$ is contained in $E_c$. This concludes the proof.

\end{proof}
\section{Vojta's conjecture}
In this section, we prove the boundedness-theorem A. 

\begin{theorem}
Let $X$ be a Hilbert modular variety with $E \subset X$ the exceptional divisor, $C$ a smooth projective curve and $S \subset C$ a finite subset. Then
$$\deg_C(f^*(K_X+E)) \leq n(2g(C)-2+s),$$
for all morphisms $f: C \to X$ such that $f^{-1}(E) \subset S$.
\end{theorem}

\begin{proof}
The morphism $f: C \to X$ induces a morphism $f': C \to \P(T_X(- \log E))$ and by definition we have an inclusion $f'^*(\mathcal{O}(1)) \hookrightarrow K_C(f^*(E)_{red})$. So we easily obtain the algebraic tautological inequality
$$\deg_C(f'^*(\mathcal{O}(1)) \leq 2g(C)-2+s.$$
Now, according to proposition \ref{NC}, $f(C)$ is not contained in a leaf of a Hilbert modular foliation.

Let $\mathcal{F}$ be one of the canonical Hilbert modular foliation on $X$. To the foliation $\mathcal{F}$ is associated a divisor $Z \subset \P(T_X(- \log E)),$ linearly equivalent to $\mathcal{O}(1)+{N}_\mathcal{F}(-E).$ Then the algebraic tautological inequality gives
$$\deg_C(f^*({N}^*_\mathcal{F}(E))) \leq \deg_C(f^*(Z))+ \deg_C(f'^*({N}^*_\mathcal{F}(E))) \leq 2g(C)-2+s.$$
The first inequality comes from the non-tangency of algebraic curves to the foliation which implies $0 \leq \deg_C(f^*(Z))$.

Now we use the tangency formula \ref{TAN}
$$N_{{\F}_1}^*\otimes...\otimes N_{{\F}_n}^*=K_X\otimes{\mathcal O}\big(-(n-1)E\big).$$
We obtain

$$\deg_C(f^*(K_X+E))=\sum_i \deg_C(f^*({N}^*_{\mathcal{F}_i}(E))) \leq n(2g(C)-2+s).$$

\end{proof}
\section{The second main theorem for entire curves into Hilbert modular varieties}
\subsection{Nevanlinna theory}
Let $X$ be a projective manifold. Nevanlinna theory is an intersection theory for holomorphic curves $f:\C \to X$ and divisors or line bundles on $X$. 

Let $L$ be a line bundle on $X$ equipped with a smooth Hermitian metric and $c_1(L)$ be its Chern form. Then the order function is defined by
$$T_f(r, L)=\int_1^r \frac{dt}{t} \int_{\Delta(t)} f^*c_1(L),$$
where $\Delta(t):=\{z \in \C | |z| <t \}.$
Once an ample line bundle $L$ is chosen, we denote $T_f(r):=T_f(r, L).$

If $D \subset X$ is an effective divisor such that $f(\C) \not \subset \operatorname{supp} D$, we define the counting function by
$$ N(r, f^*D)= \int_1^r \left(\sum_{z \in \Delta(t)} \operatorname{ord}_z f^*D \right) \frac{dt}{t},$$
and for an integer $k \geq 1$ the truncated counting function by
$$N_k(r, f^*D)= \int_1^r \left(\sum_{z \in \Delta(t)} \min (k, \operatorname{ord}_z f^*D) \right) \frac{dt}{t}.$$

The classical Nevanlinna inequality (see for example \cite{Shab}) gives
$$N(r, f^*D) \leq T_f(r, D) + O(1).$$

\subsection{The Second Main Theorem}
We will prove the following result which generalizes to any dimension the case of surfaces dealt with in \cite{Tib}.

\begin{theorem}
Consider a  projective resolution $\pi: X\rightarrow {\overline{{\H}^n / \Gamma}}$ given by Theorem \ref{TAN}, $E=\pi^{-1} (\{p_1,...,p_l\})$ the exceptional divisor, $K_X$ the canonical line bundle of $X$. Let $f: \C \to X$ be a non-constant entire curve such that $f(\C)$ is not contained in $E$. Then
$$T_f(r, K_X)+T_f(r, E) \leq nN_1(r, f^*E)+S_f(r) \|,$$
where $S_f(r)=O(\log^+ T_f(r))+o(\log r),$ and $\|$ means that the estimate holds outside some exceptional set of finite measure.
\end{theorem}

One of the main tool in the proof is the so-called Tautological inequality

\begin{theorem}[Tautological inequality \cite{McQ98}]
Let $X$ be a projective manifold, $A$ an ample line bundle, $D$ a normal crossing divisor on $X$, and $f:\C \to X$ an entire curve not contained in $D$. Let $f': \C \to \P(T_X(\log D)$ be the canonical lifting of $f$, then
$$T_{f'}(r, \O(1))\leq N_1(r, f^*D)+S_f(r) \|$$

\end{theorem}

Let us give the proof of the Second Main Theorem for Hilbert modular varieties.

\begin{proof}
Let $\mathcal{F}$ be one of the canonical Hilbert modular foliation on $X$. Let us recall that to the foliation $\mathcal{F}$ is associated a divisor $Z \subset \P(T_X(- \log E)),$ linearly equivalent to $\mathcal{O}(1)+{N}_\mathcal{F}(-E).$ Indeed, with the foliation one can consider (outside the singular locus $\operatorname{Sing}(\mathcal{F})$ which has codimension at least $2$) the exact sequence
$$0 \to T_\mathcal{F} \to T_X(- \log E) \to {N}_\mathcal{F}(-E).$$ Then one takes $Z:=\P(T_\mathcal{F}).$ We have the exact sequence over $\P(T_X(- \log E))$
$$0\to  \mathcal{O}(1)\otimes\mathcal{O}(-Z) \to \mathcal{O}(1) \to \mathcal{O}(1)_{|Z} \to 0.$$ $\mathcal{O}(1)\otimes\mathcal{O}(-Z)$ has trivial restriction to each fibre of $\pi: \P(T_X(- \log E)) \to X$ hence is of the form $\pi^*L$. Taking (the dual of) the direct image of the exact sequence, we have
$$0 \to  T_\mathcal{F} \to T_X(- \log E) \to L^* \to 0,$$ which gives that $L^*={N}_\mathcal{F}(-E).$

According to corollary \ref{entirecurvetangent}, $f: \C \to X$ cannot be tangent to $\mathcal{F}$ which means that $f'$ is not contained in $Z$.

The Tautological inequality gives the following inequality
$$T_f(r,{N}^*_\mathcal{F}(E)) \leq T_{f'}(r, Z)+T_f(r,{N}^*_\mathcal{F}(E))\leq N_1(r, f^*E)+S_f(r) \|$$

Now we use the tangency formula \ref{TAN}
$$N_{{\F}_1}^*\otimes...\otimes N_{{\F}_n}^*=K_X\otimes{\mathcal O}\big(-(n-1)E\big).$$

Combining both observations we obtain
$$T_f(r, K_X) +T_f(r, E) \leq \sum_i T_f(r,{N}^*_{\mathcal{F}_i}(E)) \leq nN_1(r, f^*E)+S_f(r) \|$$
which concludes the proof.
\end{proof}

\begin{corollary}
Let $X$ as above be a Hilbert modular variety of general type. Let $f: \C \to X$ be a non-constant entire curve which ramifies  over $E$ with order at least $n$, i.e. $f^*E\geq n \operatorname{supp} f^*E$. Then $f(\C)$ is contained in $E$.
\end{corollary}

\begin{proof}
If $f(\C)$ is not contained in $E$ then $$nN_1(r, f^*E) \leq N(r,f^*E) \leq T_f(r, E) + O(1).$$
The Second Main Theorem then gives
$$T_f(r, K_X) \leq S_f(r) \|.$$
Since $K_X$ is supposed to be a big line bundle, this gives a contradiction.
\end{proof}

\begin{remark}
It is proved in \cite{Tsu85} that if $n>6$ then $X$ is of general type.
\end{remark}
\section{A metric approach}
In this section, we follow the approach initiated in \cite{Rou13},
constructing pseudo-metric on Hilbert modular varieties. The main difference
here is that we allow elliptic fixed points.

Let $g$ denote the Bergmann metric with K\"ahler form $\omega$ on ${\H}^n$ such that $Ricci(g)=-g$ and having holomorphic sectional curvature $\leq - 1/n$. It descends to a (singular) metric on ${\H}^n / \Gamma$. The main point of this section is to extend it to $\pi: X \to\overline{{\H}^n / \Gamma}$.
 
 The exceptional divisor $E$ splits as
 
     $$E=E_c+E_e$$
     
     where the first part corresponds to the resolution of cusps and the second one to the resolution of elliptic points.

\begin{proposition}\label{Poinc}
The metric $g$ has Poincar\'e growth near $E$.
\end{proposition}

\begin{proof}
$g$ has uniformly negative holomorphic sectional curvature, so one can invoke the generalized Ahlfors-Schwarz lemma of Royden \cite{Roy}.
\end{proof}

Let $F$ be a Hilbert modular form of weight $2l$ and $\omega=dz_1\wedge \dots \wedge dz_n.$
Then $s:=F\omega^{\otimes l}$ provides a section $s \in H^0(X\setminus E, K_X^{\otimes l}).$
Therefore one can consider its norm $||s||^2_{h^l}$ with respect to the metric induced by $g$, $h=(\det g)^{-1}.$
Now, we try to find conditions on $F$ under which $||s||^{2b/ln}.g$ will extend as a pseudo-metric on $X$ for some $b>0$ suitably chosen (see proposition \ref{controlKobayashi} below).

\subsection{The case of cusps}
Consider the Bergmann metric on ${\H}^n$:

$$g=2\sum_i \frac{dz_i\otimes d\overline{z_i}}{|\Im z_i|^2}$$ with volume form
$$\Omega=i^n\prod_i \frac{dz_i\wedge d\overline{z_i}}{|\Im z_i|^2}.$$

As already described above, the change of coordinates in the resolution of cusps is given by
$$z_i=\sum_j a_{ji} \log u_j.$$

So one immediately has that $s$ extends over $E_c$ as a pluricanonical form with logarithmic poles along $E_c$, in other words as a section of $l(K_X+E)$.

Near $E_c$, the volume form becomes
$$\Omega=i^n\det (a_{ji}) \prod_i \frac{du_i\wedge d\overline{u_i}}{|u_i|^2}.\frac{1}{|\sum a_{ji} \log |u_j|^2|^2}.$$

Therefore we see that  $$||s||^{2b/ln} \leq |F|^{2b/ln}C (\log |u_1\dots u_n|^2)^{2b}. $$ 

Denote $S_k^m$ the space of Hilbert modular form of weight $k$ and order at least $m$, where the order is the vanishing order at the cusps.

Using the above proposition \ref{Poinc}, one obtains.
\begin{proposition}\label{cusp}
Let $F \in S_{2l}^{\nu l}$ and $b>0$ then $||s||^{2b/ln}.g$ extends as a pseudo-metric over cusps vanishing on $E_c$ if $\nu >\frac{n}{b}$.
\end{proposition}

\subsection{The case of elliptic singularities}
At an elliptic fixed point  the stabilizer acts on the tangent space by multiplication by $e^{2i\pi S_i}$, $i=1,\dots,n,$ with
$S_i \in \Q$, $0\leq S_i <1$.

Consider the Bergmann metric on the polydisc ${\Delta}^n$:
$$g=2\sum_i \frac{dz_i\otimes d\overline{z_i}}{(1-|z_i|^2)^2}$$
with volume form
$$\Omega=i^n\prod_i \frac{dz_i\wedge d\overline{z_i}}{(1-|z_i|^2)^2}.$$
The change of coordinates in the resolution of elliptic singularities is given by
$$z_i=\prod_l u_l^{b_{li}}.$$
So near $E_e$, the volume form becomes
$$\Omega=i^n\det (b_{li}) \prod_l \frac{du_l\wedge d\overline{u_l}}{|u_l|^2}.|u_l|^{2\sum_i b_{li}}\prod_{i}\frac{1}{(1-\prod_l|u_l|^{2b_{li}})^2} $$

\begin{proposition}\label{ell}
Let $b>0$. There is a constant $c$ depending only on the order of the stabilizer of the elliptic fixed point such that
if $F$ is a Hilbert modular form of weight $2l$ vanishing with order $c.ln$ at elliptic fixed points then $||s||^{2b/ln}.g$  extends as a pseudo-metric over elliptic singularities vanishing on $E_e$.
\end{proposition}

\begin{proof}
By the change of coordinates, we have:

$$s=F(z_1,\dots,z_n)\omega^{\otimes l}=F_0(z_1,\dots,z_n)\frac{(dz_1\wedge \dots \wedge dz_n)^{\otimes l}}{(z_1\dots z_n)^l}$$
$$=G(u_1,\dots,u_n)\frac{(du_1\wedge \dots \wedge du_n)^{\otimes l}}{(u_1\dots u_n)^l}.$$
We have $$\ord_{u_j}G= (\ord_{z_1}F+l) . b_{j1}+\dots+(\ord_{z_n}F+l) . b_{jn}\geq (c+l)(b_{j1}+\dots b_{jn}),$$
where $c:= \min \ord_{z_i}F.$

So we obtain,
$$||s||^{2b/ln} \leq C \prod_l |u_l|^{2bc \sum b_{li}}. $$ 
Therefore $||s||^{2b/ln}.g$ extends as a pseudo-metric if $2bc \sum b_{li} >2$ for all $l$.
From \cite{Tai} (Proposition 3.2), one obtains that
$$\sum b_{li} \geq m:=\min (1, \sum  S_i).$$
Therefore one can choose $c>\frac{1}{mb}.$

\end{proof}

\begin{corollary}\label{elliptic}
If 
\begin{equation}
\sum S_i\geq 1
\end{equation}

then $||s||^{2b/ln}.g$ extends as a pseudo-metric over elliptic singularities (i.e. on $E_c$) for all Hilbert modular form vanishing with order $> ln/b$ at elliptic fixed points.
\end{corollary}

\subsection{Control of the Kobayashi pseudo-distance}
Let $F$ be a Hilbert modular form and $0 < \epsilon <1/n$ such that $||s||^{2(1-n\epsilon)/ln}.g$ extends as a pseudo-metric on $X$ vanishing on $E$.

\begin{proposition}\label{controlKobayashi}
There exists a constant $\beta>0$ such that $$\tilde{g}:=\beta.||s||^{2(1-n\epsilon)/ln}.g$$ satisfies the following property: for any holomorphic map $f: \Delta \to {X}$ from the unit disc equipped with the Poincar\'e metric $g_P$, we have
$$f^*\tilde{g}\leq g_P.$$
\end{proposition}

\begin{proof}
Let $$u(z):=(1-|z|^2)^2.f^*\tilde{g}.$$ Restricting to a smaller disk, we may assume that $u$ vanishes on the boundary of $\Delta$. Take a point $z_0$ at which $u$ is maximum. $z_0 \in \Delta \setminus f^{-1}(\{s=0\})$ otherwise $u$ would vanish identically. Therefore at $z_0$ we have
$$i\partial \overline{\partial} \log u(z)_{z=z_0} \leq 0.$$
We have
$$i\partial \overline{\partial} \log u(z)=i\partial \overline{\partial} \log f^*||s||^{2(1-n\epsilon)/ln}+i\partial \overline{\partial} \log f^*g+i\partial \overline{\partial} \log (1-|z|^2)^2.$$
On $X$, we have
$$i\partial \overline{\partial} \log ||s||^{2(1-n\epsilon)/ln} = (\frac{1}{n}-\epsilon) Ricci(g)=-(\frac{1}{n} - \epsilon).g$$
by the K\"ahler-Einstein property.

$g$ has holomorphic sectional curvature $\leq -\frac{1}{n}$ therefore
$$i\partial \overline{\partial} \log f^*g \geq \frac{1}{n}.f^*g.$$
Moreover
$$i\partial \overline{\partial} \log (1-|z|^2)^2=\frac{-2}{(1-|z|^2)^2}.$$
 Combining the above inequalities, we obtain
$$ -(\frac{1}{n}-\epsilon).f^*g(z_0)+\frac{1}{n}.f^*g(z_0)-\frac{2}{(1-|z_0|^2)^2}\leq 0.$$
So
$$\epsilon/2.f^*g(z_0).(1-|z_0|^2)^2 \leq 1.$$

Finally, we have
$$u(z)\leq u(z_0)=(1-|z_0|^2)^2.\beta.||s(f(z_0))||^{2(1-n\epsilon)/ln}.f^*g(z_0)$$ $$\leq \frac{2\beta.||s(f(z_0))||^{2(1-n\epsilon)/ln}}{\epsilon}.$$
Therefore if we take
$$\beta=\frac{\epsilon/2}{\sup ||s(f(z))||^{2(1-n\epsilon)/ln}},$$
we obtain the desired property.
\end{proof}

\begin{corollary}
Let $d_X$ be the Kobayashi pseudo-distance, $\beta$ and $F$ a Hilbert modular form as above. Then $\tilde{g} \leq d_X$.
In particular, the degeneracy locus of $d_X$ is contained in the base locus of these Hilbert modular forms.
\end{corollary}

\begin{corollary}\label{SGG}
Let $X$ be a Hilbert modular variety such that there exists a Hilbert modular form as above. Then $X$ satisfies the strong Green-Griffiths-Lang conjecture.
\end{corollary}

\subsection{Existence of Hilbert modular forms}
We shall use the following result of \cite{Tsu85} (Sect. 4)
\begin{equation}\label{RR}
\dim S_k^{\nu k}(\Gamma_K) \geq (2^{-2n+1}\pi^{-2n} d_K^{3/2} \zeta_K(2)-2^{n-1}\nu^nn^{-n}d_K^{1/2}hR)k^n+O(k^{n-1})
\end{equation}
for even $k \geq 0$, where $h, d_K, R, \zeta_K$ denote the class number of $K$, the absolute value of the discriminant, the positive regulator and the zeta function of $K$.
In particular, there is a modular form $F$ with $\ord (f)/\weight(f) \geq \nu$, if
\begin{equation}\label{weight}
\nu < 2^{-3}\pi^{-2}n\left( \frac{4d_K\zeta_K(2)}{hR}\right)^{1/n}.
\end{equation}

\begin{corollary}
 For $n$ fixed, except for a finite number of $K$, there is a Hilbert modular form $F$ such that  $||s||^{2(1-n\epsilon)/ln}.g$ extends as a pseudo-metric over cusps. 
\end{corollary}

\begin{proof}
If we fix $n$, then $\zeta_K(2)$ has a positive lower bound independent of $K$. Since $hR \sim d_K^{1/2}$ by the Brauer-Siegel Theorem, for any constant $C$ there are only a finite number of $K$ such that the right hand side of (\ref{weight}) is smaller than $C$. One concludes with proposition \ref{cusp}.
\end{proof}

Moreover as $d_K$ tends to infinity, the coefficient of the leading term of (\ref{RR})  grows at least with order $O( d_K^{3/2}).$

\begin{corollary}
If the number of elliptic fixed points is $O(d_K^{\epsilon})$ for $0<\epsilon<3/2$, then with finite exceptions, Hilbert modular varieties of dimension $n$ satisfy the strong Green-Griffiths-Lang conjecture.
\end{corollary}

\begin{proof}
If $d_K$ is large enough, one can find modular forms satisfying the hypotheses of proposition \ref{cusp} and \ref{ell}. One can therefore apply corollary \ref{SGG}
\end{proof}

\subsection{Estimation of elliptic fixed points}
Prestel \cite{Prestel} has obtained precise formula on the number of elliptic points of the Hilbert modular group.
In particular, one can deduce

\begin{proposition}
For fixed $n$, the number of equivalence classes of elliptic fixed points is $O(d_K^{\frac{1}{2}+\epsilon})$ for every $\epsilon>0$.
\end{proposition}

\begin{proof}
We follow the approach of (\cite{Tsu86}, p.664) to estimate the number of elliptic points of trace $s$ usually denoted $l(s,1)$.
Recall that $h, d_K, R$ denote the class number of $K$, the absolute value of the discriminant and the positive regulator. Let $K'=K(\sqrt{s^2-4}).$ Let $h',d_{K'}, R'$ denote the class number of $K'$, the absolute value of the discriminant and the positive regulator and let $\mathcal{D}_{K'/K}$ be the relative discriminant. Then $\mathcal{D}_{K'/K} | 4-s^2.$ Let $\mathcal{U}_0$ be the ideal in $O_K$ determined by $\mathcal{U}_0^2\mathcal{D}_{K'/K}=(4-s^2)O_K.$
By Prestel \cite{Prestel} (5.4) the number of elliptic points with trace $s$ is
$$l(s,1)=\frac{R'h'}{Rh}.\sum_{\mathcal{U}|\mathcal{U}_0} \frac{\phi_{O_K'}(\mathcal{U}_0\mathcal{U}^{-1}O_{K'})}{\phi_{O_K}(\mathcal{U}_0\mathcal{U}^{-1})},$$
where $\phi_{O_K}(\mathcal{U}_0\mathcal{U}^{-1})$ denotes the number of prime residue classes of $O_K$ mod $\mathcal{U}_0\mathcal{U}^{-1}.$
Now $$\phi_{O_K'}(A) \leq \phi_{O_K}(A)^2$$ for any ideal $A$ in $O_K$ since $K'$ is at most a quadratic extension of $K$.
Moreover, $$\phi_{O_K}(A) \leq N(A)$$ where $N(A)$ denotes the norm.

Therefore, a rough estimate gives
$$\sum_{\mathcal{U}|\mathcal{U}_0} \frac{\phi_{O_K'}(\mathcal{U}_0\mathcal{U}^{-1}O_{K'})}{\phi_{O_K}(\mathcal{U}_0\mathcal{U}^{-1})} \leq \sum_{\mathcal{U}|\mathcal{U}_0} N(\mathcal{U}) \leq N(\mathcal{U}_0)^2 \leq N(4-s^2). $$
which is constant for fixed $n$. By the Brauer-Siegel theorem (\cite{Lang} Corollary p.328) $\log (hR)\sim \log d_K^{1/2}.$ Since $d_{K'}=|N_{K/\Q}(\mathcal{D}_{K'/K})|d_K^2$ we obtain that for any $\epsilon>0,$ $$l(s,1)<d_K^{1/2+\epsilon}$$ for $d_K$ sufficiently large. 
To finish let us remark that, $n$ being fixed, there are only finitely many possibilities for $s$.
\end{proof}

\begin{corollary}
If $n\geq 2$ then with finite exceptions, Hilbert modular varieties of dimension $n$ satisfy the strong Green-Griffiths-Lang conjecture.
\end{corollary}

\bibliography{bibliography}{}

\end{document}